\documentclass[12pt]{article}
\usepackage{amssymb,amsmath,graphics,amsthm}
\usepackage{color}
\newtheorem{thm}{Theorem}
\newtheorem{lemma}[thm]{Lemma}

\newtheorem{conj}{Conjecture}

\newcommand{\zt}{{\mathbb Z}^2}
\renewcommand{\P}{{\mathbb P}}
\newcounter{mycount}
\newenvironment{mylist}{\begin{list}{{\rm (\roman{mycount})}}%
{\usecounter{mycount}\itemsep 0pt}}{\end{list}}

\newcommand{\dof}{\bf}

\title{Local Bootstrap Percolation}
\author{
Janko Gravner\thanks{Funded in part by NSF Grant DMS-0204376
and the Republic of Slovenia's Ministry of Science Program P1-285}
 \and
Alexander E. Holroyd\thanks{Funded in part by an NSERC (Canada)
Discovery Grant, and by Microsoft Research} }

\begin{document}
\maketitle
\renewcommand{\thefootnote}{}
\footnotetext{{\bf\hspace{-6mm}Key words:} bootstrap percolation,
cellular automaton, metastability, finite-size scaling, crossover}
\footnotetext{{\bf\hspace{-6mm}\bf 2000 Mathematics Subject
Classifications:} Primary 60K35; Secondary 82B43}

\begin{abstract}
We study a variant of bootstrap percolation
in which growth is restricted to a single active cluster.
Initially there is a single {\em active} site at the origin, while
other sites of $\mathbb{Z}^2$ are independently {\em occupied}
with small probability $p$, otherwise {\em empty}.  Subsequently,
an empty site becomes active by contact with two or more active
neighbors, and an occupied site becomes active if it has an active
site within distance 2. We prove that the entire lattice becomes
active with probability $\exp [\alpha(p)/p]$, where $\alpha(p)$ is
between $-\pi^2/9+c\sqrt p$ and $-\pi^2/9+C\sqrt p(\log
p^{-1})^3$. This corrects previous numerical predictions for the
scaling of the correction term.  
\end{abstract}

\section{Introduction}

{\dof Local bootstrap percolation} is a 3-state cellular automaton
on the square lattice $\zt$ defined as follows.   At each time
step $t=0,1,2\ldots$, each site of $\zt$ is either {\dof empty}
($\circ$), {\dof occupied} ($\bullet$), or {\dof active}
($\star$). Let $p\in (0,1)$.  The {\em initial} configuration is
given by a random element $\sigma$ of
$\{\circ,\bullet,\star\}^{\zt}$ under a probability measure $\P_p$
in which
\begin{align*}
\P_p\big[\sigma(0)=\star\big]=p,& \quad \P_p\big[\sigma(0)=\circ\big]=1-p;\\
\P_p\big[\sigma(x)=\bullet\big]=p,& \quad
\P_p\big[\sigma(x)=\circ\big]=1-p, \quad (x\neq 0);
\end{align*}
and the states of different sites are independent.  (Here
$0=(0,0)\in\zt$ is the origin).  Subsequently, the configuration
at time $t+1$ is obtained from that at $t$ according to the
following deterministic rules.
\begin{itemize}
\setlength{\leftmargin}{1cm} \setlength{\itemsep}{-1ex}
\item[(L1)] Each $\bullet$ becomes $\star$ if it has at least one
$\star$ within $\ell^1$-distance 2.
\item[(L2)] Each $\circ$ becomes $\star$ if it has at least two
$\star$'s within $\ell^1$-distance 1.
\item[(L3)] All other states remain unchanged.
\end{itemize}
We are interested in {\dof indefinite growth}, i.e., the event
that every site in $\zt$ eventually becomes active.  The following
is our main result.

\begin{thm}[growth probability]
\label{main} There exists constants $c,C,p_0>0$ such that for
$p<p_0$,
$$
\exp\frac{-2\lambda+c\sqrt p}{p}
\leq \P_p\big(\text{\rm indefinite growth}\big)
\leq \exp\frac{-2\lambda+C\sqrt p\,(\log p^{-1})^3}{p},
$$
where $\lambda=\pi^2/18$.
\end{thm}

\paragraph{Bootstrap percolation.}
As suggested by its name, the process we have introduced is a
close relative of (standard) {\dof bootstrap percolation}.  This
well-studied model may be described in our setting as follows.
Initially, each site in the square $\{1,\dots, L\}^2$ is
independently active with probability $p$ and empty otherwise, and
all sites outside the square are empty.  The configuration then
evolves according to rules (L2,L3) ((L1) being irrelevant), and
one is interested in the probability $I(L,p)$ that the entire
square eventually becomes active.

Bootstrap percolation has been studied both rigorously and
numerically \cite{adler-vis, a-s-a, aizenman-lebowitz,van-enter},
both in its own right and as a tool to analyze other models
\cite{martinelli, dlbd, dld, fontes-schonmann-sidoravicius,
frobose}. An important property is that, asymptotically as $p\to
0$, it undergoes a sharp metastability transition as the parameter
$p\log L$ crosses the threshold $\lambda$.  More precisely, as
proved in \cite{h-boot}, $I(L,p)$ converges to $0$ or $1$ if
respectively $p\log L <\lambda-\epsilon$ or $p\log L
>\lambda+\epsilon$ (also see \cite{h-boot, h-l-r} for similar
results on related models).  However, simulations for moderate
values of $p$ show surprising discrepancies with this rigorous
result \cite{dlbd,dld,gh-slow,stauffer}, and our motivation in
this article is to advance understanding of this latter
phenomenon.

Local bootstrap percolation has been an implicit ingredient in
many arguments involving bootstrap percolation, including those in
\cite{aizenman-lebowitz, gh-slow, h-boot}, and a variant form appears
explicitly in \cite{dlbd,dld} (see the discussion below).  The
proof of the main result from \cite{h-boot} already implies that
for the local model, the quantity  
$$- p\,\log\P_p(\text{indefinite growth})$$
converges to $2\lambda$ as $p\to 0$, and the
main contribution of this paper is to identify the speed of this
convergence up to logarithmic factors.

Heuristically, the dominant mechanism for active sites to take
over space in bootstrap percolation is the presence of widely
separated, and essentially independent, local bootstrap
percolations. Indeed, for small $p$, bootstrap percolation
immediately fixates (stops changing) on the overwhelming
proportion of the lattice; the rare nuclei that facilitate growth
hence encounter configurations of fixated active sites, which
should be not very different from occupied sites in the local
version. This suggests that $I(L,p)$ makes the transition from
near 0 to near 1 as $L^2\,\P_p(\text{indefinite growth})$ changes
from small to large, and thus when $2p\log L \approx -p\log
\P_p(\text{indefinite growth})$. In this subject, claims not
backed by rigorous arguments have a questionable track record;
nevertheless, motivated by Theorem \ref{main}, we venture the
following conjecture.
\begin{conj}\label{conj}
In the standard bootstrap percolation model, let $L\to\infty$ and
$p\to 0$ simultaneously. Then,
with $\lambda=\pi^2/18$ as above and any $\epsilon>0$, we have
$$
\text{if}\quad L<\exp\frac{\lambda-p^{1/2-\epsilon}}{p} \qquad
\text{then}\quad I(L,p)\to 0.
$$
\end{conj}

The complementary bound, namely,
$$
\text{if}\quad L>\exp \frac{\lambda-c\,p^{1/2}}{p} \qquad
\text{then}\quad I(L,p)\to 1,
$$
for a small enough $c>0$,
was proved in \cite{gh-slow}, so Conjecture \ref{conj} states that
the power of $p$ in the correction term is exactly $1/2$.

\paragraph{Rectangle process.}
It is natural to consider the following variant of the growth
model defined by (L1--L3), in which we update states in a
different order.  A {\dof rectangle} is a set of sites of the form
$R=\{a,\ldots,c\}\times\{b,\ldots,d\}\subset\zt$.  The {\dof
rectangle process} is a (random) sequence of rectangles
$\rho_0\subseteq\rho_1\subseteq\rho_2\subseteq\cdots$ defined in
terms of the initial configuration $\sigma$ as follows.  If
$\sigma(0)=\circ$ then we set $\rho_i=\emptyset$ for all $i$.  If
$\sigma(0)=\star$ then we set $\rho_0=\{0\}$, and then proceed
inductively as follows.  If
$\rho_i=\{a,\ldots,c\}\times\{b,\ldots,d\}$, then consider the
configuration in which every site in $\rho_i$ is active, every
site outside the enlarged rectangle
$\rho_i^+:=\{a-2,\ldots,c+2\}\times\{b-2,\ldots,d+2\}$ is empty,
and all sites in $\rho_i^+\setminus\rho_i$ have the same state as
in $\sigma$.  Apply the update rules (L1--L3) to this
configuration until the configuration stops changing.  It is
readily seen that the resulting set of active sites is again a
rectangle; let $\rho_{i+1}$ be this set.

It is straightforward to check that the set of eventually active
sites in the local bootstrap percolation model is identical to the
limiting rectangle $\lim_{i\to\infty} \rho_i$, and in particular
indefinite growth occurs if and only if the latter equals $\zt$.

The rectangle process can be described via a
count\-able-state-space Mar\-kov chain, whose state represents the
current rectangle together with information about which sites on
its sides have been examined.  In principle, this allows for a
computational approach to estimating the probability of indefinite
growth; as we do not presently pursue this, we will not give a
detailed description of the chain.

\paragraph{Variant models.}
In addition to the standard model, our methods adapt to the
following two variants.  In the {\dof modified local model}, we
replace rule (L1) with:
\begin{trivlist}
\item(L1F) Each $\bullet$ becomes $\star$ if it has at least one
$\star$ within $\ell^\infty$-distance 1.
\end{trivlist}
In the {\dof Frob\"ose local model}, we replace (L1) with:
\begin{trivlist}
\item(L1M) Each $\bullet$ becomes $\star$ if it has at least one
$\star$ within $\ell^1$-distance 1.
\end{trivlist}
These models may be regarded as local versions of {\dof modified
bootstrap percolation} \cite{h-boot,h-mod,schonmann} (in which
an empty site becomes active if it has at least one active
horizontal neighbor and at least one active vertical neighbour),
and {\dof Frob\"ose bootstrap percolation} \cite{frobose} (in
which activation of a site requires a horizontal neighbor,
a vertical neighbor, and the diagonal neighbor between
them all to be active).
\begin{thm}[growth probability for variant models]
\label{main-fm} Theorem \ref{main} holds for the modified and
Fro\-b\"o\-se local models, but with $\lambda=\pi^2/6$.  For the
Frob\"ose model the upper bound can be improved to
$\exp[(-2\lambda+C\sqrt p\,(\log p^{-1})^2)/p]$.
\end{thm}
The proof of Theorem \ref{main-fm} follows the same steps as that
of Theorem \ref{main} (with a few simplifications).  We therefore
omit the details and instead summarize the differences in Section
\ref{sec:mod}.

\paragraph{Comparison with numerical results.}
For the modified and Frob\"ose models, the resulting rectangle
process, together with the associated Markov chain, is much
simpler, which makes a computational approach more inviting.
Indeed, in \cite{dld}, using computer calculations together with
heuristic arguments, the authors obtained the estimates
$$
-p\,\log\P_p(\text{indefinite growth}) \approx 2\lambda-6.22\,
p^{0.333}
$$
for the modified local model, and
$$
-p\,\log\P_p(\text{indefinite growth}) \approx 2\lambda- 5.25\,
p^{0.388}
$$
for the Frob\"ose local model, as $p\to 0$.  These estimates seems
to fit well down to $p\approx 0.01$, but Theorem \ref{main-fm}
contradicts them in the limit, because the estimated powers 0.333
and 0.388 are less than 1/2.

In a similar vein, in \cite{stauffer}, using interpolation between
simulations and the rigorous result of \cite{h-boot}, it was
estimated for the standard (non-local) bootstrap percolation
model, that the metastability transition occurs at
$$p\log L \approx \lambda - 0.51\,p^{0.2}.$$
This would again be inconsistent in the limit $p\to 0$ with our
Conjecture \ref{conj}.

These deceptively simple models thus present yet another puzzle of
the type ``why are we not able to see the asymptotic behavior in
simulations?''  That computer simulations can be so misleading
\cite{h-boot,van-enter} is arguably the primary lesson learned
from more than two decades of research into bootstrap percolation
by mathematicians and physicists, and the present article is
another contribution to this theme.

\paragraph{Outline of proof.}
We conclude the introduction with a brief outline of the rest of
the paper.  The bulk of the work, contained in Section
\ref{sec:ub}, is in proving the upper bound in Theorem \ref{main}. We
consider the rectangle process as the side length grows from
$1/\sqrt p$ to $p^{-1}\log p^{-1}$.  We sample this process in a
sequence of coarse-grained steps, in which the side length
increases by a factor of roughly $1+\sqrt p$ at each step.  These
step sizes are chosen to balance the entropy factors (related to
the number of possibilities), and errors arising from the
replacing true probabilities of growth with larger quantities. The
latter are probabilities of ``no double gap'' conditions for
growth from a smaller to a larger rectangle, as in \cite{h-boot}.
One complication, contributing to the logarithmic error factor
between our upper and lower bounds, is that corner sites between
the two rectangles may allow growth in two directions.

The rectangle process can grow via rectangles of unequal
dimensions; however, by the variational principle of
\cite{h-boot}, the most likely scenario is the most symmetric one,
i.e., growth of squares. This yields the value of the constant
$\lambda$.

The lower bound in Theorem \ref{main} is a straightforward
consequence of methods from \cite{gh-slow}, which we summarize in
Section \ref{sec:lb}.  The idea is to consider paths in the space
of rectangle sizes that deviate from symmetric growth by
approximately $1/\sqrt p$ at scale $1/p$ (and are therefore not
prohibitively improbable).  The increased entropy from such paths
is enough to introduce an additional multiplicative factor of
$\exp(c/\sqrt p)$ to the bound on the growth probability.

\section{Upper Bound}
\label{sec:ub}

It will be convenient to assume throughout that the probability
$p$ is sufficiently small, e.g., $p<0.1$  suffices.  Then we
denote
\begin{align*}
q=q(p)&:=-\log(1-p); \\
A=A(q)&:=\lceil 1/\sqrt q \rceil; \\
B=B(q)&:=\lfloor q^{-1}\log q^{-1}\rfloor.
\end{align*}
Note that $q=p+p^2/2+p^3/3+\cdots$ for small $p$.  Throughout, we
use $c, c_1,c_2,\dots$ and $C, C_1, C_2,\dots$ to denote positive
(respectively small and large) absolute constants, which may in
principle be explicitly computed. For a rectangle
$R=\{a,\ldots,c\}\times\{b,\ldots,d\}$ we denote its {\dof
dimensions} $\text{dim}(R)=(c-a+1,d-b+1)$.

We will bound the probability of indefinite growth by summing over
possible trajectories for the rectangle process.  However, the
total number of trajectories is too large, so we need the
following concept. Take a sequence of rectangles $R_1, \dots,
R_{n+1}$, where $n\geq 1$. Write $\text{dim}(R_i)=(a_i,b_i)$ and
$s_i:=a_{i+1}-a_i$ and $t_{i}:=b_{i+1}-b_i$. We call the sequence
{\dof good} if it has the following properties. \pagebreak
\begin{mylist}
\item $0\in R_1\subseteq R_2 \subseteq\dots\subseteq R_{n+1}$.
\item $\min(a_1,b_1)\in [A, A+3]$.
\item $a_n+b_n\le B$.
\item $a_{n+1}+b_{n+1}>B$.
\item For $i=1,\dots, n$, either
$s_i\ge a_i\sqrt q$ or $t_i\ge b_i\sqrt q$.
\item For $i=1,\dots, n$, both
$s_i< a_i\sqrt q+4$ and $t_i< b_i\sqrt q+4$.
\end{mylist}

We next define some useful events. The {\dof columns} of a
rectangle $R=\{a,\ldots,c\}\times\{b,\ldots,d\}$ are the sets
$\{a\}\times\{b,\ldots,d\},\ldots, \{c\}\times\{b,\ldots,d\}$.  We
say that $R$ has a {\dof double gap in the columns} if two
consecutive columns are entirely empty in the {\em initial}
configuration $\sigma$.  Double gaps in the {\dof rows} are
defined similarly.  For a rectangle $R$, define the event
$$
G(R):=\{R \text{ has no double gaps in the columns or rows}\}.
$$

Moreover, we recall the following definition from \cite{h-boot}.
For two rectangles $R\subseteq R'$, define subrectangles
$S_1,\ldots ,S_8$ (some of which may be empty) according to Figure
\ref{r1r8}, so that $R'$ is the disjoint union of
$R,S_1,\dots,S_8$. Define $D(R,R')$ to be the event that each of
the two rectangles $S_1\cup S_8 \cup S_7$ and $S_3\cup S_4\cup
S_5$ has no double gaps in the columns, and each of the two
rectangles $S_1\cup S_2 \cup S_3$ and $S_7\cup S_6\cup S_5$ has no
double gaps in the rows. The idea is that this event is necessary
for the growth to proceed from $R$ to $R'$.
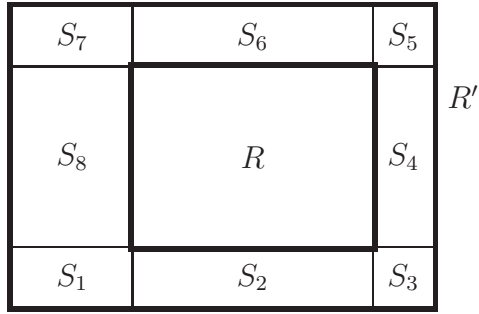
\begin{figure}
\begin{center}
\setlength{\unitlength}{8mm}
\begin{picture}(7,5)(0,0)
\put(2,0){\line(0,1){5}}
\put(6,0){\line(0,1){5}}
\put(0,1){\line(1,0){7}}
\put(0,4){\line(1,0){7}}
\linethickness{1.6pt}
\put(0,0){\framebox(7,5){}}
\put(2,1){\framebox(4,3){$R$}}
\put(0,0){\makebox(2,1){$S_1$}}
\put(2,0){\makebox(4,1){$S_2$}}
\put(6,0){\makebox(1,1){$S_3$}}
\put(6,1){\makebox(1,3){$S_4$}}
\put(6,4){\makebox(1,1){$S_5$}}
\put(2,4){\makebox(4,1){$S_6$}}
\put(0,4){\makebox(2,1){$S_7$}}
\put(0,1){\makebox(2,3){$S_8$}}
\put(6.5,3){\makebox(2,1){$R'$}}
\end{picture}
\end{center}
\caption{The rectangles $S_1,\ldots, S_8$.}
\label{r1r8}
\end{figure}

Finally, define the event
$$
\begin{aligned}
E:=\bigcup\bigg\{G(R):\;
\begin{array}{l}
R \text{ contains $0$ and has one dimension in} \\
\text{$[B-A-10, B-A]$ and the other in $[1,A]$}
\end{array}
\bigg\}.
\end{aligned}
$$

\begin{lemma}[key inclusion]
\label{keyinclusion} If indefinite growth occurs, then either $E$
occurs, or else there exists a good sequence $R_1,\dots, R_{n+1}$
of rectangles such that
$$G(R_1) \cap \bigcap_{i=1}^n D(R_i,R_{i+1})$$
occurs.
\end{lemma}

\begin{proof}
We start with the following observation: if the rectangle process
ever encounters a given rectangle $R$, then $G(R)$ occurs.  To see
this, suppose on the contrary that $R$ has two consecutive vacant
columns, say.  If these columns contain the origin then no growth
occurs.  Otherwise, so long as the growing rectangle includes no
site above or below $R$, it contains no site in the two columns.
This follows by induction on the steps of the rectangle process
(no $\circ$ in the two columns can have two adjacent $\star$'s
within $R$, while no $\bullet$ outside $R$ can contribute to a
site becoming $\star$ without first becoming $\star$ itself).

Now define the {\dof good region\/} for dimensions of rectangles
to be
$$
T=\big\{(a,b)\in\zt: a,b\ge A \text{ and } a+b\le B\big\}
$$
(see Figure \ref{gs}). Note that in the sequence of dimensions
$(\dim(\rho_j))_{j\geq 0}$ for the rectangle process, each
coordinate increases by at most $4$ at each step, while if
indefinite growth occurs then at least one coordinate increases by
at least $1$ at every step, and both coordinates tend to $\infty$.
\begin{figure}[t]
\begin{center}
\input{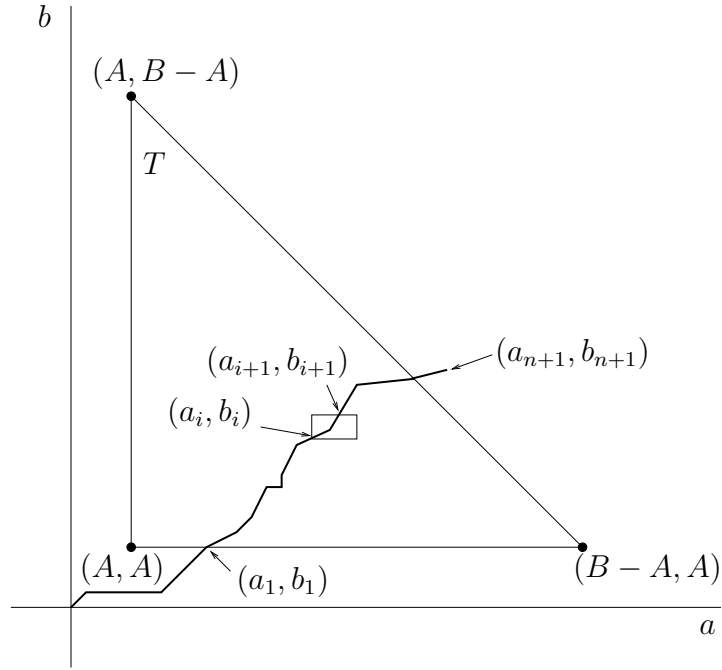}
\end{center}
\caption{The sequence of dimensions $(a,b)$ of the rectangle
process (bold curve), as it passes through the good region $T$.
The small rectangular window has dimensions $(a_i\sqrt q, b_i\sqrt
q)$.} \label{gs}
\end{figure}

Therefore, assuming indefinite growth, if $(\dim(\rho_j))_{j\geq
0}$ never enters $T$, then $E$ must occur, since the sequence must
``escape'' near a corner of the good region.  On the other hand,
if the sequence does enter $T$, then let $R_1$ be the first
$\rho_j$ such that $\dim(\rho_j)\in T$.  Then property (ii) of a
good sequence will be satisfied.  Thereafter, define $R_2,\ldots,
R_{n+1}$ iteratively as follows.  Assuming $\dim(R_i)=(a_i,b_i)$,
let $R_{i+1}$ be the first rectangle encountered by the rectangle
process after $R_i$ such that either $s_i\geq a_i\sqrt q$ or
$t_i\geq b_i\sqrt q$ (where $(s_i,t_i)=\dim(R_{i+1})-\dim(R_i)$).
This ensures (v),(vi) are satisfied.  Continue in this way,
stopping at the first such rectangle, $R_{n+1}$, whose dimensions
are outside $T$.  Then (iii),(iv) are satisfied.

The rectangles thus constructed form a good sequence. From the
above observations, $\cap_{i=1}^{n+1} G(R_i)$ occurs, and since
$D(R,R')\subseteq G(R')$ the result follows.
\end{proof}

\vfill\eject

Next we bound the probabilities of the events appearing in Lemma
\ref{keyinclusion}.  We introduce the following functions from
\cite{h-boot}:
$$\beta(u)=\frac{u+\sqrt{u(4-3u)}}{2} \quad\text{and}\quad
g(z)=-\log \beta(1-e^{-z}).
$$
The threshold $\lambda$ is determined by the integral (see
\cite{h-boot}, and also \cite{h-l-r})
\begin{equation}\label{integral}
\int_0^\infty g(z)\;dz = \lambda=\frac{\pi^2}{18}.
\end{equation}

\begin{lemma}[double gaps]
\label{basicbound}
If $R$ is a rectangle with $\dim(R)=(a,b)$,
$$
\P_p\big(R \text{\rm\ has no double gaps in the columns}\big)\le
\exp\big[-(a-1)g(bq)\big].
$$
For rows we have the same bound with $a$ and $b$ exchanged.
\end{lemma}

\begin{proof}
See \cite[Lemma 8(i)]{h-boot}.
\end{proof}

The following is an enhanced version of \cite[Proposition
21]{h-boot}. The reader may want to consult that proof for more
details.

\begin{lemma}[border event]
\label{border}
For rectangles $R\subseteq R'$ of dimensions $(a,b)$ and
$(a+s,b+t)$,
$$
\begin{aligned}
&\P_p\big[D(R,R')\big] \\
&\le \exp \Big(-sg(bq)-tg(aq)+
2[g(bq)+g(aq)]+stq\, e^{2[g(bq)+g(aq)]}\Big).
\end{aligned}
$$
\end{lemma}
We remark that the second appearance of the constant 2 is crucial;
the lemma would not hold with a smaller constant, while a larger
constant would not suffice for the remainder of our argument,
owing to the behaviour of $g$ near zero (Lemma \ref{gbound}
below).

\begin{proof}[Proof of Lemma \ref{border}]
We first claim that
$$
\P_p\big[D(R,R')\big]\le \sum_{k=0}^{st}
{\binom{st}{k}} p^k(1-p)^{st-k}e^{-(s-2k-2)g(bq)}e^{-(t-2k-2)g(aq)}.
$$
To prove this, we split the event according to the set of occupied
sites in the set $S_1\cup S_3\cup S_5\cup S_7$. If the number of
such sites is $k$, they divide $S_4$ and $S_8$ into at most $k+2$
contiguous vertical strips, each of which has no double gaps in
the columns.  A similar argument applies to rows in $S_2$
and $S_6$, and the relevant events are independent.  Then use
Lemma \ref{basicbound}.

From the above, dropping the power of $(1-p)$ and using the
binomial expansion, we obtain
$$
\P_p\big[D(R,R')\big]\le e^{-sg(bq)-tg(aq)+2[g(bq)+g(aq)]}
\left(1+p e^{2[g(bq)+g(aq)]}\right)^{st}
$$
Applying $1+z\le e^z$ and $p\le q$ to the last factor finishes the
proof.
\end{proof}

\begin{lemma}[escape probability]
\label{Ebound} For some absolute constant $c>0$ and all $p< 0.1$,
$$
\P_p(E)\le \exp\left[ -c q^{-1}(\log q^{-1})^2\right].
$$
\end{lemma}
\begin{proof}
By Lemma \ref{basicbound} and the definition of $E$,
$$
\P_p(E)\le 11 A^2 B\cdot \exp[(B-A-11)g(Aq)].
$$
The factor in front of the exponential comes from
the number of possible choices of rectangles. Now use
the definitions of $A$ and $B$ and the fact that
\[
g(\epsilon)\sim \tfrac 12 \log\epsilon^{-1} \quad\text{as }
\epsilon\to 0.
\qedhere\]
\end{proof}

We have provided all the probabilistic bounds we need, and we now
proceed to analytic estimates of the expressions that appear.  The
following will be applied to the first two terms in the
exponential in Lemma \ref{keyinclusion}.

\begin{lemma}[variational principle]
\label{varpri} Let $A,B$ be any integers with $A>4$ and $B>2A$,
and suppose $(a_i,b_i)_{i=1,\ldots,n+1}$ satisfy $a_0=b_0=A$, and
$s_i:=a_{i+1}-a_i\ge 0$ and $t_i:=b_{i+1}-b_i\ge 0$, and
properties (ii)--(iv) in the definition of a good sequence.  For
any $q>0$ and any positive, smooth, convex, decreasing function
$g:(0,\infty)\to(0,\infty)$,
$$
\sum_{i=0}^n\big[s_ig(b_iq)+t_ig(a_iq)\big]\ge
\frac 2q\int_{Aq}^{Bq}g-2Bg(Bq/2).
$$
\end{lemma}

\begin{proof}
We make use of the following technology from \cite[Section
6]{h-boot}. If $\gamma$ is a piecewise-linear path in the quadrant
$(0,\infty)^2$, with each segment oriented in the non-negative
direction of both co-ordinates, we define
$$w(\gamma):=\int_\gamma \Big( g(y)dx + g(x)dy\Big).$$
Let $\gamma_1$ be a piecewise-linear path with vertices
$(a_0q,b_0q),\ldots,(a_{n+1}q, b_{n+1}q)$ (in that order). Since
$g$ is decreasing, the sum in the statement of the lemma is at
least $w(\gamma_1)/q$ (compare \cite[Proposition 16]{h-boot}).

Also let $\gamma_2$ be the straight line path from $(a_{n+1}q,
b_{n+1}q)$ to $(Bq,Bq)$, and let $\Delta$ be the straight path
from $(Aq,Aq)$ to $(Bq,Bq)$. The variational principle in
\cite[Lemma 16]{h-boot} states that $w$ is minimized by paths that
follow the main diagonal, therefore
$$
w(\gamma_1)+w(\gamma_2)\ge w(\Delta).
$$
However, we have
$$w(\Delta)=2 \int_{Aq}^{Bq}g.$$
On the other hand,
$$
w(\gamma_2)\le 2Bq\cdot g(Bq/2),
$$
because $\min\{a_{n+1},b_{n+1}\}> B/2$ and $(B-a_{n+1})+(B-b_{n+1})<2B$.
\end{proof}

\begin{lemma}[integral]
\label{integralest} With $q,A,B,g$ as defined previously, for $p<
0.1$ and an absolute constant $C_1$,
$$
\int_{Aq}^{Bq} g\ge \frac{\pi^2}{18}-C_1{\sqrt q}\log q^{-1}.
$$
\end{lemma}

\begin{proof}
Use (\ref{integral}) together with the asymptotics
\[
\int_0^\epsilon g\;\stackrel{\epsilon\to 0}{\sim}\;
 \frac 12\epsilon\log \epsilon^{-1}
\quad\text{and}\quad
\int_K^\infty g\;\stackrel{K\to \infty}{\sim}\; \frac 12e^{-2K}.
\qedhere\]
\end{proof}

\begin{lemma}
\label{Bgbound}
For some $C_2$ and all $p< 0.1$ we have
$
Bg(Bq/2)\le C_2\log q^{-1}.
$
\end{lemma}
\begin{proof}
As $g(K)\sim e^{-2K}$ as $K\to\infty$, and by the definition of $B$,
\[
Bg(Bq/2)\le \frac 1q\log q^{-1} C' \cdot e^{-2\cdot \frac 2q
(q^{-1}\log q^{-1}-1)}.
\qedhere\]
\end{proof}

The next two lemmas will be used to bound the last two terms in
the exponential in Lemma \ref{keyinclusion}.
\begin{lemma}
\label{gbound}
For every $a\le B$ we have
$
e^{2g(aq)}\le \frac {C_3}{aq}\, \log q^{-1}.
$
\end{lemma}

\begin{proof}
We have $e^{g(z)}=1/\beta(1-e^{-z})$. Moreover,
$\beta(1-e^{-z})\sim z^{1/2}$ as $z\to 0$ and
$\beta(1-e^{-z})\to 1$ as $z\to\infty$.
It follows that, for a large enough $M$,
$$
\sup_{0<z\le M}\frac{\sqrt z}{\beta(1-e^{-z})}\le 2\sqrt M.
$$
Therefore,
\[
e^{2g(aq)}\le \left( \frac {C' \sqrt{Bq}}{\sqrt{aq}}\right)^2.
\qedhere\]
\end{proof}

\begin{lemma}[summation bound]
\label{sumbound} Let $n$ and $a_i, b_i$ $(i=1,\dots, n+1)$ be
positive integers and suppose that $s_i:=a_{i+1}-a_i\ge 0$ and
$t_i:=b_{i+1}-b_i\ge 0$ for $i=1,\dots, n$. Further, assume that
(ii)--(vi) in the definition of a good sequence are satisfied.
Then
$$
n\le \frac{1}{\sqrt q}\log q^{-1}
\quad\text{and}\quad
\sum_{i=1}^{n}\frac {s_it_i}{a_ib_i}\le {C_4}{\sqrt q}\log q^{-1}.
$$
\end{lemma}

\begin{proof}  To bound $n$, use (v),(ii) and (iii) to get
$$
(1+ \sqrt q)^n \le \frac {a_nb_n}{a_1b_1}\le \frac {B^2}{A^2}\le
\frac 1q(\log q^{-1})^2.
$$
Then use this to prove the second bound as follows. As $a_i\ge A$
and $b_i\ge A$ and $A\sqrt  q\ge 1$, (vi) implies $s_i<5a_i\sqrt
q$ and $t_i<5b_i\sqrt q$. Therefore
\[
\sum_{i=1}^n \frac{s_it_i}{a_ib_i}\le n\cdot 25 q.
\qedhere\]
\end{proof}

It will be convenient to bound the probability of $G(R_1)$
(associated with the entry of the rectangle process into the good
region) in terms of a factor that can be easily combined with
those obtained in Lemma \ref{border}, in order to use Lemma
\ref{varpri}.  The following lemma deals with this estimate.
\begin{lemma}[entry term]
\label{startingpart} Let $R_1, \dots, R_{n+1}$ be a good sequence
of rectangles and let $a_0=b_0:=A$, with $s_0:=a_1-a_0$ and
$t_0:=b_1-b_0$. Then, for some absolute constant $C_5$,
$$\P_p\big[G(R_1)\big]\le
\exp\big[-s_0g(b_0q)-t_0g(a_0q)+C_5q^{-1/2}\log q^{-1}\big].$$
\end{lemma}

\begin{proof}
Without loss of generality suppose $a_1\ge b_1$.
By Lemma \ref{basicbound},
$$
P[G(R_1)]\le e^{-(a_1-1)g(b_1q)}.
$$
Now $b_1\le A+3$. We have $g'(z)\sim 1/(2\sqrt z)$ as $z\to 0$,
and $g$ is convex. Therefore
$$
0\le g(Aq)-g(b_1q)\le 3q\cdot g'(Aq)\le C_6\sqrt q.
$$
Also $t_0\leq 3$, so
$$
e^{t_0g(a_0q)}\le e^{3g(Aq)}\le C_7 q^{-3/2}.
$$
Therefore,
\[
\frac{\P_p[G(R_1)]}{\exp[-s_0g(b_0q)-t_0g(a_0q)]}\le C_7 q^{-3/2}
\cdot e^{C_6 B\sqrt q}.
\qedhere\]
\end{proof}

\begin{lemma}[entropy]
\label{entropy}
The number of good sequences of rectangles is at most
$$
\exp\left[\frac {C_8}{\sqrt q}(\log q^{-1})^2\right].
$$
\end{lemma}
\begin{proof}
There are at most $B^4$ choices for $R_1$.  For each subsequent
rectangle there are at most $(B\sqrt q+4)^4$ choices, corresponding
to the distance by which the rectangle grows in each of the four
directions. The number of steps is $n+1$, which is bounded by
Lemma \ref{sumbound}.  This gives the following upper bound:
\[
B^4\left(B\sqrt q +4\right)^{4(q^{-1/2}\log q^{-1}+1)}.
\qedhere\]
\end{proof}

We are now ready to prove the upper bound in the main theorem.
\begin{proof}[Proof of Theorem \ref{main} (upper bound)]
Take $p<0.1$. For a good sequence
of rectangles, the events $G(R_1),D(R_i,R_{i+1})$ appearing in
Lemma \ref{keyinclusion} are all independent, since they involve
the initial states of disjoint sets of sites. Therefore,
\begin{equation}
\label{bigsum}
\P_p\big(\text{indefinite growth}\big)\le \P_p(E)+
\sum_{R_1,\dots, R_{n+1}}
\P_p\big[G(R_1)\big]\prod_{i=1}^{n} \P_p\big[D(R_i, R_{i+1})\big],
\end{equation}
where the sum is over all possible good sequences of rectangles.
We now proceed to bound the various terms.

We start by using Lemmas \ref{border} and \ref{gbound} to bound
$\P_p[D(R_i, R_{i+1})]$ above by
$$
\begin{aligned}
&\exp\left[ -s_ig(b_iq)-t_ig(a_iq)\right]
\left(\frac
{C_3}{Aq}\log q^{-1}\right)^2 \exp\bigg[s_it_i q\frac
{C_3}{a_iq}\log q^{-1}\frac
{C_3}{b_iq}\log q^{-1}\bigg]\\
\le &
\exp\left[-s_ig(b_iq)-t_ig(a_iq)\right]
\cdot
\frac {C_9}{q}\left(\log q^{-1}\right)^2
\cdot
\exp\left[\frac{C_9}{q}\left(\log q^{-1}\right)^2
\frac{s_it_i}{a_ib_i} \right].
\end{aligned}
$$
Using this and Lemma \ref{startingpart} we next bound
$
\P_p[G(R_1)]\prod_{i=1}^{n} \P_p[D(R_i, R_{i+1})]
$
above by
$$
\begin{aligned}
&\exp\bigg(-\sum_{i=0}^n\big[ s_ig(b_iq)+t_ig(a_iq)\big]\bigg)\\
\times &
\exp\left[\frac{C_5}{\sqrt q}\log q^{-1}\right]
\cdot
\bigg(\frac {C_9}{q}\left(\log q^{-1}\right)^2\bigg)^n
\cdot
\exp\bigg[\frac{C_9}{q}\left(\log q^{-1}\right)^2
\sum_{i=1}^n\frac{s_it_i}{a_ib_i} \bigg].
\end{aligned}
$$
By Lemma \ref{sumbound}, this expression is at most
$$
\exp\bigg(-\sum_{i=0}^n[s_ig(b_iq)+t_ig(a_iq)]\bigg)
\cdot
\exp\left[\frac{C_{10}}{\sqrt q}(\log q^{-1})^3\right].
$$
Next, we use Lemmas \ref{varpri}, \ref{integralest} and
\ref{Bgbound}, to bound
$$
\sum_{i=0}^n\big[ s_ig(b_iq)+t_ig(a_iq)\big]\ge
\frac {2\lambda}q-\frac {C_{11}}{\sqrt q}\log q^{-1}.
$$

Finally, we substitute the last two bounds,
together with Lemmas \ref{entropy} and \ref{Ebound},
into (\ref{bigsum}) to get
\begin{align*}
\P_p(\text{indefinite growth})
\quad\le\quad\exp\left[ -c q^{-1}(\log q^{-1})^2\right]
\qquad\qquad\qquad\qquad\qquad\\
+\quad \exp\left[\frac {C_8}{\sqrt q}(\log q^{-1})^2\right]
\cdot
\exp\left[-\frac{2\lambda}{q} +\frac{C_{11}}
{\sqrt{q}}\log q^{-1}\right]
\cdot
\exp\left[\frac{C_{10}}{\sqrt q}(\log q^{-1})^3\right].
\end{align*}
Therefore
$$
\P_p(\text{indefinite growth})\le
\exp\left[-\frac{2\lambda}{q}+\frac{C'}{\sqrt q}(\log q^{-1})^3\right].
$$
As a final step, use the asymptotics $q=p+\mathcal O(p^2)$ as
$p\to 0$ to replace $q$ with $p$.
\end{proof}

\section{Lower Bound}
\label{sec:lb}

\begin{proof}[Proof of Theorem \ref{main} (lower bound)]
This follows easily from the proof of \linebreak Proposition 8 in
\cite{gh-slow}. In that article, it is proved that there is a
certain event ${\cal F}$, satisfying
$$\P_p({\cal F})\geq \exp[-2\lambda/p+c/\sqrt p],$$
for small $p$, and defined in terms of the initial states in the quadrant
$Q:=\{0,1,\ldots\}^2\subset\zt$, on which the entire quadrant $Q$
eventually becomes active in the bootstrap percolation model.  On
the same event (but replacing $\star$ with $\circ$ for sites
$x\neq 0$), the entire quadrant becomes active in the local
bootstrap percolation model also.  This fact would not follow for
an arbitrary event, but it holds for the particular event ${\cal
F}$ because it is tailored to produce growth starting from the
origin.

For the reader's convenience we briefly summarize the construction
of the event ${\cal F}$, referring to \cite{gh-slow} for the
details. We consider two ways to grow from active square $[0,a]^2$
to the active square $[0,b]^2$. The first, symmetrical, way is
that the top and right sides of the growing square always each
have an occupied site within distance 2.  The alternative,
deviant, possibility is that vertical growth is temporarily halted
by two adjacent vacant rows, and the square first elongates
horizontally to the rectangle $[0,b]\times [0,a]$, when it can
finally proceed to grow vertically until it reaches
$[0,b]^2$. A key estimate, \cite[Lemma 13]{gh-slow}, states
roughly that, for $a,b$ on the scale $1/p$, the quotient of
probabilities of the second and the first event is at least
$c_1\,p\,\exp\left[-2C\,p(b-a)^2\right]$. This suggests that one
should keep $b-a$ of order $1/\sqrt p$ to incur a deviation cost
$c_2 p$. Indeed, one can make about $m:=c_3 /\sqrt p$ deviant
steps within these constraints, and a careful organization ensures
that different choices of these steps give disjoint events. The
resulting lower bound on $\P_p({\cal F})$ is greater than
$\exp\left[-2\lambda/p\right]$ by at least the factor
$$
{\binom{p^{-1}}{m}}\big(1/\sqrt p\big)^m\cdot (c_2 p)^m\approx
(c_2/c_3)^m,
$$
which gives the required bound provided we choose $c_3\ll c_2$.

Finally let ${\cal G}$ be the event that every semi-infinite line
of sites, started at arbitrary site and moving in any direction
parallel to one of the axes, contains at least one $\bullet$
initially. Clearly, $\P_p({\cal G})=1$ and
$$
\mathcal{F}\cap\mathcal{G}\subset\{\text{indefinite growth}\},
$$
completing the proof.
\end{proof}

\section{Modified and Fr\"obose models}
\label{sec:mod}

In this section we describe the proof of
Theorem \ref{main-fm}. The proof follows the same steps, and in
fact a few simplifications are possible.  We therefore summarize
the differences.

For both the modified and Fr\"obose models, we replace the
condition of ``no double gaps'' in the definitions of the events
$G(R)$ and $D(R,R')$ with the condition that {\em no} columns (or
rows) are initially entirely empty.  The bound in Lemma
\ref{basicbound} then becomes $\exp[-af(bq)]$, where
$$f(z)=-\log(1-e^{-z}).$$
The function $f$ then replaces $g$ throughout, and the threshold
$\lambda$ then arises as $\int_0^\infty f=\pi^2/6$.  We also
modify the definition of the rectangle process so that $\rho^+$ is
$\rho$ enlarged by only $1$ in each direction.

For the modified model, the change to Lemma \ref{basicbound}
allows the second $2$ in the exponential in Lemma \ref{border} to
be replaced with $1$. In fact this change is needed for the
argument to go through, because of the different behavior of $f$
near zero. Specifically, Lemma \ref{gbound} holds with $e^{f(aq)}$
on the left side (and the same right side).

For the Frob\"ose model we can make a further simplification to
the argument, as horizontal and vertical progress now occur
disjointly. To be more precise, $D(R,R')$ is now defined to be the
event that the rectangle process makes a transition from $R$ to
$R'$. Then, referring to Figure 1, the two events:
\begin{gather*}
\text{$S_1\cup S_2\cup S_3$ and $S_5\cup S_6\cup S_7$
have no empty rows, and} \\
\text{$S_1\cup S_7\cup S_8$ and $S_3\cup S_4\cup S_5$
 have no empty columns,}
\end{gather*}
occur disjointly and are thus negatively correlated by the Van den
Berg Kesten inequality (see e.g.\ \cite{g2}).  This fact
eliminates the last summand in the exponential in Lemma
\ref{border}, makes Lemma \ref{gbound} unnecessary and reduces
the power of $\log p^{-1}$ to 2 in the final bound. 

\section*{Open Problems}

\begin{mylist}
\item Is a power of $\log p^{-1}$ in the upper bound of
Theorem \ref{main} really necessary?
\item  In the modified and Frob\"ose models, can the discrepancy
between the asymptotic power $1/2$ and the numerical estimates,
which are closer to $1/3$, be explained?
\item Is it possible to make the same arguments work if the distance in
rule (L1) is altered to something larger than 2?
\end{mylist}

\bibliography{local.bbl}

\section*{}
\vspace{-10mm} \sc

\noindent Alexander E. Holroyd: {\tt holroyd{\char64}math.ubc.ca}\\
Department of Mathematics, University of British Columbia,\\
121-1984 Mathematics Rd, Vancouver, BC V6T 1Z2, Canada \\

\noindent Janko Gravner: {\tt gravner{\char64}math.ucdavis.edu}\\
Mathematics Department, University of California,\\
Davis, CA 95616, USA

\end{document}